\documentclass[12pt]{amsart}
\usepackage{amsmath,amssymb,amsbsy,amsfonts,latexsym,amsopn,amstext,cite,
                                               amsxtra,euscript,amscd,bm}
\usepackage{url}

\usepackage{mathrsfs}
\usepackage{color}
\usepackage[colorlinks,linkcolor=blue,anchorcolor=blue,citecolor=blue,backref=page]{hyperref}
\usepackage{color}
\usepackage{graphics,epsfig}
\usepackage{graphicx}
\usepackage{float}
\usepackage{epstopdf}
\hypersetup{breaklinks=true}

\usepackage[np]{numprint}
\npdecimalsign{\ensuremath{.}}

\usepackage{bibentry}

\usepackage[english]{babel}
\usepackage{mathtools}
\usepackage{todonotes}
\usepackage{url}
\usepackage[colorlinks,linkcolor=blue,anchorcolor=blue,citecolor=blue,backref=page]{hyperref}

\usepackage[norefs,nocites]{refcheck}

\begin{document}

\newtheorem{theorem}{Theorem}
\newtheorem{lemma}[theorem]{Lemma}
\newtheorem{claim}[theorem]{Claim}
\newtheorem{cor}[theorem]{Corollary}
\newtheorem{conj}[theorem]{Conjecture}
\newtheorem{prop}[theorem]{Proposition}
\newtheorem{definition}[theorem]{Definition}
\newtheorem{question}[theorem]{Question}
\newtheorem{example}[theorem]{Example}
\newcommand{\hh}{{{\mathrm h}}}
\newtheorem{remark}[theorem]{Remark}

\numberwithin{equation}{section}
\numberwithin{theorem}{section}
\numberwithin{table}{section}
\numberwithin{figure}{section}

\def\sssum{\mathop{\sum\!\sum\!\sum}}
\def\ssum{\mathop{\sum\ldots \sum}}
\def\iint{\mathop{\int\ldots \int}}

\newcommand{\diam}{\operatorname{diam}}

\def\squareforqed{\hbox{\rlap{$\sqcap$}$\sqcup$}}
\def\qed{\ifmmode\squareforqed\else{\unskip\nobreak\hfil
\penalty50\hskip1em \nobreak\hfil\squareforqed
\parfillskip=0pt\finalhyphendemerits=0\endgraf}\fi}

\newfont{\teneufm}{eufm10}
\newfont{\seveneufm}{eufm7}
\newfont{\fiveeufm}{eufm5}
%
%
\newfam\eufmfam
     \textfont\eufmfam=\teneufm
\scriptfont\eufmfam=\seveneufm
     \scriptscriptfont\eufmfam=\fiveeufm
%
%
\def\frak#1{{\fam\eufmfam\relax#1}}

\newcommand{\bflambda}{{\boldsymbol{\lambda}}}
\newcommand{\bfmu}{{\boldsymbol{\mu}}}
\newcommand{\bfxi}{{\boldsymbol{\eta}}}
\newcommand{\bfrho}{{\boldsymbol{\rho}}}

\def\eps{\varepsilon}

\def\fK{\mathfrak K}
\def\fT{\mathfrak{T}}
\def\fL{\mathfrak L}
\def\fR{\mathfrak R}

\def\fA{{\mathfrak A}}
\def\fB{{\mathfrak B}}
\def\fC{{\mathfrak C}}
\def\fM{{\mathfrak M}}
\def\fS{{\mathfrak  S}}
\def\fU{{\mathfrak U}}

\def\T {\mathsf {T}}
\def\Tor{\mathsf{T}_d}
\def\Tore{\widetilde{\mathrm{T}}_{d} }

\def\sM {\mathsf {M}}

\def\Kmnd{\cK_d(m,n)}
\def\Kmnp{\cK_p(m,n)}
\def\Kmnq{\cK_q(m,n)}

\def \balpha{\bm{\alpha}}
\def \bbeta{\bm{\beta}}
\def \bgamma{\bm{\gamma}}
\def \bdelta{\bm{\delta}}
\def \bzeta{\bm{\zeta}}
\def \blambda{\bm{\lambda}}
\def \bchi{\bm{\chi}}
\def \bphi{\bm{\varphi}}
\def \bpsi{\bm{\psi}}
\def \bnu{\bm{\nu}}
\def \bomega{\bm{\omega}}

\def \beta{\bm{\eta}}

\def \bell{\bm{\ell}}

\def\eqref#1{(\ref{#1})}

\def\vec#1{\mathbf{#1}}

\newcommand{\abs}[1]{\left| #1 \right|}

\def\Zq{\mathbb{Z}_q}
\def\Zqx{\mathbb{Z}_q^*}
\def\Zd{\mathbb{Z}_d}
\def\Zdx{\mathbb{Z}_d^*}
\def\Zf{\mathbb{Z}_f}
\def\Zfx{\mathbb{Z}_f^*}
\def\Zp{\mathbb{Z}_p}
\def\Zpx{\mathbb{Z}_p^*}
\def\cM{\mathcal M}
\def\cE{\mathcal E}
\def\cH{\mathcal H}

\def\le{\leqslant}

\def\ge{\geqslant}

\def\sfB{\mathsf {B}}
\def\sfC{\mathsf {C}}
\def\L{\mathsf {L}}
\def\FF{\mathsf {F}}

\def\sE {\mathscr{E}}
\def\sS {\mathscr{S}}

\def\cA{{\mathcal A}}
\def\cB{{\mathcal B}}
\def\cC{{\mathcal C}}
\def\cD{{\mathcal D}}
\def\cE{{\mathcal E}}
\def\cF{{\mathcal F}}
\def\cG{{\mathcal G}}
\def\cH{{\mathcal H}}
\def\cI{{\mathcal I}}
\def\cJ{{\mathcal J}}
\def\cK{{\mathcal K}}
\def\cL{{\mathcal L}}
\def\cM{{\mathcal M}}
\def\cN{{\mathcal N}}
\def\cO{{\mathcal O}}
\def\cP{{\mathcal P}}
\def\cQ{{\mathcal Q}}
\def\cR{{\mathcal R}}
\def\cS{{\mathcal S}}
\def\cT{{\mathcal T}}
\def\cU{{\mathcal U}}
\def\cV{{\mathcal V}}
\def\cW{{\mathcal W}}
\def\cX{{\mathcal X}}
\def\cY{{\mathcal Y}}
\def\cZ{{\mathcal Z}}
\newcommand{\rmod}[1]{\: \mbox{mod} \: #1}

\def\cg{{\mathcal g}}

\def\vy{\mathbf y}
\def\vr{\mathbf r}
\def\vx{\mathbf x}
\def\va{\mathbf a}
\def\vb{\mathbf b}
\def\vc{\mathbf c}
\def\vh{\mathbf h}
\def\vk{\mathbf k}
\def\vm{\mathbf m}
\def\vz{\mathbf z}
\def\vu{\mathbf u}
\def\vv{\mathbf v}

\def\e{{\mathbf{\,e}}}
\def\ep{{\mathbf{\,e}}_p}
\def\eq{{\mathbf{\,e}}_q}

\def\Tr{{\mathrm{Tr}}}
\def\Nm{{\mathrm{Nm}}}

 \def\SS{{\mathbf{S}}}

\def\lcm{{\mathrm{lcm}}}

 \def\0{{\mathbf{0}}}

\def\({\left(}
\def\){\right)}
\def\l|{\left|}
\def\r|{\right|}
\def\fl#1{\left\lfloor#1\right\rfloor}
\def\rf#1{\left\lceil#1\right\rceil}
\def\sumstar#1{\mathop{\sum\vphantom|^{\!\!*}\,}_{#1}}

\def\mand{\qquad \mbox{and} \qquad}

\def\tblue#1{\begin{color}{blue}{{#1}}\end{color}}




\hyphenation{re-pub-lished}

\mathsurround=1pt

\def\bfdefault{b}

\def \F{{\mathbb F}}
\def \K{{\mathbb K}}
\def \N{{\mathbb N}}
\def \Z{{\mathbb Z}}
\def \Q{{\mathbb Q}}
\def \R{{\mathbb R}}
\def \C{{\mathbb C}}
\def\Fp{\F_p}
\def \fp{\Fp^*}

 \def \xbar{\overline x}

\title[Hausdorff dimension of large values  of Weyl sums]{Hausdorff dimension of the large values  of Weyl sums}

 \author[C. Chen] {Changhao Chen}

\address{Department of Pure Mathematics, University of New South Wales,
Sydney, NSW 2052, Australia}
\email{changhao.chenm@gmail.com}

 \author[I. E. Shparlinski] {Igor E. Shparlinski}

\address{Department of Pure Mathematics, University of New South Wales,
Sydney, NSW 2052, Australia}
\email{igor.shparlinski@unsw.edu.au}

%

\begin{abstract}  The authors   have recently obtained a   lower bound of the Hausdorff dimension for the sets of vectors  $(x_1, \ldots, x_d)\in [0,1)^d$ 
with large Weyl sums, namely of vectors for which 
$$
\left| \sum_{n=1}^{N}\exp\(2\pi i \(x_1 n+\ldots +x_d n^{d}\)\) \right| \ge  N^{\alpha}
$$
for infinitely many integers $N \ge 1$. 
 Here we obtain an upper bound for the Hausdorff dimension of these exceptional sets.
\end{abstract}

 
\maketitle

\section{Introduction}

\subsection{Motivation and background}

For an integer $d \geqslant 2$, let $\Tor = (\R/\Z)^d$ be the  $d$-dimensional unit torus. 

For  a vector $\vx = (x_1, \ldots, x_d)\in \Tor$ and $N \in\N$, we consider the exponential   
sums
$$
S_d(\vx; N)=\sum_{n=1}^{N}\e\(x_1 n+\ldots +x_d n^{d} \), 
$$
which  are commonly called {\it Weyl sums\/}, where   throughout  the paper we denote $\e(x)=\exp(2 \pi i x)$.

The authors~\cite[Appendix~A]{ChSh1} have  shown  that  for almost all $\vx\in \Tor$ (with respect to Lebesgue measure) one has   
 \begin{equation}
\label{eq:M-R}
  |S_d(\vx; N)| \leqslant  N^{1/2+o(1)} \text{ as } N\rightarrow \infty, 
\end{equation}
see also~\cite[Theorem~2.1]{ChSh2} for a different proof.  It is very natural to  conjecture that the exponent $1/2$ is the best possible value, and indeed for $d=2$  the  authors~\cite{ChSh3} proved that for almost all $(x_1, x_2)\in \T_2$  we have  
$$
\limsup_{N\rightarrow \infty}\left|S_2((x_1, x_2); N)\right| N^{-1/2} (\log\log N)^{-1/6}=\infty.
$$ 
 However there seems to be  no results in this direction for $d\ge 3$.


For integer $d\ge 2$  and $0<\alpha<1$ our main  object is defined as  
$$
\cE_{d, \alpha}=\{\vx\in \Tor:~|S_d(\vx; N)|\geqslant N^{\alpha} \text{ for infinitely many } N\in \N\}.
$$ 
We can restate the bound~\eqref{eq:M-R} in the following way: for any $\alpha\in (1/2, 1)$ the set $\cE_{d, \alpha}$ is of Lebesgue measure zero. 
Here we are mostly interested in the structure of the sets $\cE_{d, \alpha}$, and for convenience we 
call the set $\cE_{ d, \alpha}$ the exceptional set for any integer $d\ge 2$ and  each $0<\alpha<1$. 

The authors~\cite{ChSh1} show that in terms of the Baire categories and Hausdorff dimension the exceptional sets $\cE_{d, \alpha}$ are quite massive. 
By~\cite[Theorem~1.3]{ChSh1},  for each $0<\alpha<1$ and integer $d\ge 2$ the  set $\Tor\setminus \cE_{d, \alpha}$ is of the first \emph{Baire category}. Alternatively,  this is equivalent to the statement that the complement 
$\Tor\setminus \Xi_{d}$ to the set 
\begin{equation}
\label{eq:xid}
\begin{split}
\Xi_{d}= \bigl\{\vx\in \Tor:~\forall \varepsilon>0, |S_d(\vx; N)| &\geqslant N^{1-\varepsilon}\\
& \text{for infinitely many } N\in \N \bigr\}
\end{split}
\end{equation}
is of first category, see~\cite{ChSh1} for more details and reference therein. For the Hausdorff dimension it is shown in~\cite[Theorem~1.5]{ChSh1} that for any $d\ge 2$ and $0<\alpha<1$ one has 
\begin{equation}
\label{eq:dim}
\dim \cE_{d, \alpha}\ge \xi(d, \alpha)>0
\end{equation}
with some explicit constant $\xi(d, \alpha)$.  
 
 We remark that the authors~\cite[Corollary~1.9]{ChSh2} have  obtained a nontrivial upper bound for the Hausdorff  dimension of $\cE_{d, \alpha}$ for some $\alpha$, however the bounds there  are not fully explicit and  do not cover the whole range $1/2<\alpha<1$. 
 

Here  we obtain the nontrivial upper bound of $\dim \cE_{d, \alpha}$ for all $1/2<\alpha<1$ and $d\ge 2$. 
 

On the other hand, we note that we do not have any 
plausible conjecture about the exact value  of  the  Hausdorff dimension of $\cE_{d, \alpha}$.

\subsection{Main results}

For  $\cA\subseteq \R^d$, the $s$-dimension {\it Hausdorff measure\/} of $\cA$ is defined as 
$$
\mathcal{H}^{s}(\cA)=\lim_{\delta\rightarrow 0} \mathcal{H}^{s}_{\delta}(\cA),
$$
where
$$
\mathcal{H}^{s}_{\delta}(\cA)=\inf\left \{ \sum_{i=1}^{\infty}\(\diam\cU_i\)^{s}:~\cA\subseteq \bigcup_{i=1}^{\infty} \cU_i \text{ and } \diam\cU_i \le \delta,\ i\in \N \right \}.
$$
The {\it Hausdorff dimension\/} of $\cA$ is defined as 
\begin{align*}
\dim \cA& =\inf\{s>0:~ \mathcal{H}^{s}(\cA)=0\}\\
&=\sup\{s>0:~ \mathcal{H}^{s}(\cA)=\infty\}.
\end{align*}
 We refer to~\cite{Falconer} for more details and properties  of Hausdorff dimension.

For integer $d\ge 2$ and $0<\alpha<1$  denote 
\begin{equation}
\label{eq:def u} 
\mathfrak{u}(d, \alpha)=\min_{k=0, \ldots, d-1} \frac{(2d^{2}+4d)(1-\alpha)+k(k+1)}{4-2\alpha+2k}.
\end{equation}

\begin{theorem}
\label{thm:dim}
For any integer $d\ge 2$ and $0<\alpha<1$ we have 
$$
 \dim \cE_{d, \alpha}\le \mathfrak{u}(d, \alpha).
$$ 
\end{theorem}

For  $d\ge 2$ and any $1/2<\alpha<1$  an elementary calculation gives  that $\mathfrak{u}(d, \alpha)<d$. In fact by
taking  $ k=d-1$ in~\eqref{eq:def u}   we derive 
$$
 \mathfrak{u}(d, \alpha) \le d -\frac{d(d+1)(2 \alpha -1)}{2(d+1 -\alpha)}. 
$$

Thus,   we have

\begin{cor}
For any integer $d\ge 2$ and any $1/2<\alpha<1$ we have 
$
\dim \cE_{d, \alpha} <d.
$ 
\end{cor}
Furthermore taking, for example, $k =0$ in~\eqref{eq:def u} we obtain 
$$
\dim\cE_{d, \alpha}\le \mathfrak{u}(d, \alpha)\le \frac{(2d^{2}+4d)(1-\alpha)}{4-2\alpha}.
$$

We note that although the lower  bound~\eqref{eq:dim} and the upper bound of Theorem~\ref{thm:dim} are 
of very different magnitude with respect to $d$, however for $\alpha\to 1$ they give the same rate of convergency to zero  which 
of  order $1-\alpha$. 
More precisely, the explicit formula for $ \xi(d, \alpha)$ from~\cite{ChSh1} and the formula~\eqref{eq:def u} yield
$$
c_1(d) \le \liminf_{\alpha \to 1}\, (1-\alpha)^{-1}\dim \cE_{d, \alpha}\le\limsup_{\alpha \to 1} \, (1-\alpha)^{-1}\dim \cE_{d, \alpha}\le c_2(d)
$$
for two positive constants $c_1(d), c_2(d)$ depending only on $d$. In fact for $d=2$ we have
$$ 
c_1(2) =  3
\mand 
c_2(2) = 8,
$$
while for $d\ge 3$ we have 
$$ 
c_1(d) =  \max_{\nu =1, \ldots, d} \min\left\{ \frac{1}{\nu} ,  \frac{2}{2d-\nu} \right \}
\mand 
c_2(d) = d^2 + 2d.
$$

In particular, we have

\begin{cor}
For any integer $d\ge 2$, if $\alpha\to 1$ then $\dim \cE_{d, \alpha} \rightarrow 0$. 
\end{cor}

From the definition of $\Xi_d$, see~\eqref{eq:xid}, we have  $\Xi_d\subseteq \cE_{d, \alpha}$ for any $0<\alpha<1$.  Therefore 

\begin{cor}
For and integer $d\ge 2$, we have
$\dim \Xi_d =0$.
\end{cor}

\section{Preliminaries}
 \subsection{Notation and conventions}

Throughout the paper, the notation $U = O(V )$, 
$U \ll V$ and $ V\gg U$  are equivalent to $|U|\leqslant c|V| $ for some positive constant $c$, 
which throughout the paper may depend on the degree $d$ and occasionally on the small real positive 
parameters $\varepsilon$ and $\delta$.

For any quantity $V> 1$ we write $U = V^{o(1)}$ (as $V \to \infty$) to indicate a function of $V$ which 
satisfies $|U| \le V^\eps$ for any $\eps> 0$, provided $V$ that is large enough.   
 
 We use $\# \cX$ to denote the cardinality of set $\cX$.  
 
 We always identify $\Tor$ with half-open unit cube $[0, 1)^d$, in particular we
 naturally associate Euclidean norm  $\|x\|$ with points $x \in \Tor$. Moreover we always assume that $d\ge 2$.

We say that some property holds for almost all $\vx \in \Tor$ if it holds for a set 
 $\cX \subseteq \Tor$ of  Lebesgue measure  $\lambda(\cX) = 1$. 
 
We always keep the subscript $d$ in notations for our main objects of interest such as 
$\cE_{ d, \alpha}$, $S_d(\vx; N)$ and $\Tor$, but sometimes suppress
it in auxiliary quantities.

\subsection{Mean value theorems} 

The  {\it Vinogradov mean value theorem\/}  in the currently known form, 
due to Bourgain, Demeter and Guth~\cite{BDG} for $d \geqslant 4$
and Wooley~\cite{Wool2} for $d=3$,  asserts that,
$$
\int_{\Tor} |S_d(\vx; N)|^{2s(d)}d\vx \leqslant  N^{s(d)+o(1)},
$$
where $s(d) = d(d+1)/2$.  We will use the following  result due to Wooley~\cite[Theorem~1.1]{Wool5}, which extends the bound to the Weyl sums with weights. 

\begin{lemma}   
\label{lem:Wooley}
For  any sequence of complex  weights $\vec{a} = (a_n)_{n=1}^\infty$, and 
any integer $N\ge 1$,   we have the upper bound  
$$
\int_{\Tor} |\sum_{n=1}^N a_n\e(x_1n+\ldots+x_dn^d) |^{2s(d) } d\vx\le N^{o(1)} \(\sum_{n=1}^{N}|a_n|^2\)^{s(d)}.
$$
\end{lemma}

\subsection{Completion method}
\label{subsec:completion}

The following  bound is special case of~\cite[Lemma~3.2]{ChSh2}, and  for completeness we give a proof here. 
\begin{lemma}\label{lem:control} 
For $\vx\in \Tor$ and $1\le M\le N$ we have 
$$
S_{d}( \vx; M) \ll W_{d}( \vx; N),
$$
where 
$$
W_{d}( \vx; N)= \sum_{h=1}^{N} \frac{1}{h} \left| \sum_{n=1}^{N}   \e(h n/N) \e\(x_1n+\ldots+x_dn^d \) \right|.
$$
\end{lemma}
\begin{proof} For $\vx\in \Tor$ and $n\in \N$ denote
$$
f(n)=x_1 n+\ldots + x_dn^d.
$$

Observe that by the orthogonality 
$$
\frac{1}{N}\sum_{h=1}^{N}\sum_{k=1}^{M} \e\left(h(n-k)/N\right)=
  \begin{cases}
   1 & n=1, \ldots, M, \\
  0 & \text{otherwise}.
  \end{cases}
$$
We also note that for $1 \le h,M \le N$ we have
$$
  \sum_{k=1}^{M}\e\left(hk/N\right )  \ll \frac{N}{\min\{h, N+1 - h\}},
$$
see~\cite[Equation~(8.6)]{IwKow}.  
It follows that 
\begin{align*}
S_d(\vx; M)
&=\sum_{n=1}^{N}\e(f(n)) \frac{1}{N}\sum_{h=1}^{N}\sum_{k=1}^{M} \e\left(h(n-k)/N\right)\\
&=\frac{1}{N} \sum_{h=1}^{N}  \sum_{k=1}^{M}\e\left(-hk/N\right ) \sum_{n=1}^{N} \e(f(n))\e\left (hn/N\right ) \\
&\ll \sum_{h=1}^{N} \frac{1}{\min\{h, N+1 - h\}} \left |\sum_{n=1}^{N} \e\left (hn/N\right )\e(f(n)) \right |\\
&\ll \sum_{h=1}^{N} \frac{1}{h} \left |\sum_{n=1}^{N} \e\left (hn/N\right )\e(f(n)) \right |,
\end{align*}  
which finishes the proof.
\end{proof}

Observe that for any $N$ there exists a sequence $b_N(n), n=1, \ldots, N$ such that 
\begin{equation}
\label{eq:bn}
b_N(n) \ll \log N, \qquad n=1, \ldots, N, 
\end{equation}
and $W_{d}( \vx; N)$ can be written as 
\begin{equation}
\label{eq:W}
W_{d}( \vx; N)=\sum_{n=1}^N b_N(n)  \e(x_1n+\ldots + x_d n^d).
\end{equation}

From Lemma~\ref{lem:control} we immediately obtain:

\begin{cor}
\label{cor:key}
Let $ 0<\alpha<1$ and $N_i=2^i, i \in \N$. Using above notation for any  $\eta>0$ we have 
$$
\cE_{d, \alpha+\eta} \subseteq \{\vx\in \Tor:~|W_{d}(\vx; N_i)|\ge N_i^{\alpha} \text{ for infinitely many } i\in \N\}.
$$
\end{cor}

Thus for the purpose of estimate the set $\cE_{d, \alpha}$ it is sufficient to know the size of the set
$$
\{\vx\in \Tor:~|W_{d}(\vx; N)|\ge N^{\alpha} \}, 
$$
which we investigate in Section~\ref{subsec:distribution} below.

\subsection{Distribution of large values of exponential sums}
\label{subsec:distribution}
 
We first remark that the results in this subsection are  special forms of~\cite{ChSh2}, see also~\cite[Lemma~2.1]{Wool3}.  
For completeness we give proofs for these special cases.

For $\vu\in \R^d$ and $\bzeta=(\zeta_1, \ldots, \zeta_d)$ with  $\zeta_j>0$, $j=1, \ldots, d$, we define the  $d$-dimensional rectangle (or box) with the centre $\vu$ and the side lengths $2\bzeta$ by 
$$
\cR(\vu, \bzeta)=[u_1-\zeta_1, u_1+\zeta_1)\times \ldots \times [u_d-\zeta_d, u_d+\zeta_d).
$$

In analogue of~\cite[Lemma~2.1]{Wool3} and~\cite[Lemma~3.5]{ChSh2} we obtain:

\begin{lemma} 
\label{lem:cont} 
Let $0<\alpha<1$ and let $\varepsilon>0$ be  sufficiently small. If 
$|W_{d}(\vx; N)|\ge N^{\alpha}$ for  some $\vx\in \Tor$,  then 
$$
|W_{d}( \vy; N)|\ge  N^{\alpha}/2
$$
holds for  any $\vy\in \cR(\vx, \bzeta)$ provided  that $N$ is large enough and 
$$
0<\zeta_j\le N^{\alpha-j-1-\eps},  \qquad j =1, \ldots, d.
$$
\end{lemma}
\begin{proof} For any $h=1, \ldots, N$ we have 
\begin{align*}
\sum_{n=1}^{N}   \e(h n/N) & \left(\e\(x_1n+\ldots+x_dn^d \) - \e\(y_1n+\ldots+y_dn^d \) \right)\\
& \quad\quad\quad\quad\quad\quad \ll \sum_{n=1}^{N} \sum_{j=1}^d \zeta_j n^j\le N^{\alpha-\eps/2}.
\end{align*}
The last estimate holds for all large enough $N$. By Lemma~\ref{lem:control} we obtain 
$$
|W_{d}( \vx; N)-W_d(\vy; N)| \ll N^{\alpha-\eps/2} \log N \le N^{\alpha}/2,
$$
which holds for all large enough $N$ and gives the result.
\end{proof}

In analogue of~\cite[Lemma~3.7]{ChSh2}  from Lemmas~\ref{lem:control} and~\ref{lem:cont} we obtain:

\begin{lemma}
\label{lem:counting}
Let $0<\alpha<1$ and $\eps>0$ be a small parameter. For each $j=1, \ldots, d$ let 
$$
\zeta_j=1/ \rf{N^{j+1+\eps-\alpha}}.
$$
We divide $\Tor$ into 
$$
U = \prod_{j=1}^d \zeta_j^{-1}
$$ 
boxes of the type
$$
[n_1\zeta_1, (n_1+1)\zeta_1)\times \ldots \times [n_d\zeta_d, (n_d+1)\zeta_d),
$$
where $n_j=0, \ldots, 1/\zeta_j-1$,  $j=1, \ldots, d$. Let $\fR$ be the collection of these boxes, and  
$$
\widetilde \fR=\{\cR \in \fR:~\exists\, \vx\in \cR \text{ with } |W_d(\vx; N)|\ge N^{\alpha}\}.
$$
Then one has
$$
\# \widetilde \fR
 \le U N^{s(d)(1-2\alpha)+o(1)}.
$$
\end{lemma}
\begin{proof}
Let $\cR\in \fR$.  By Lemma~\ref{lem:cont}  if $|W_{d}( \vx; N)|\ge N^{\alpha}$ for some $\vx\in \cR$, 
then   for any $\vy\in \cR$ we have $|W_{d}(\vy; N)|\ge N^{\alpha}/2$. 
Combining with Lemma~\ref{lem:Wooley}  and~\eqref{eq:bn}, ~\eqref{eq:W} we derive   
$$
N^{ 2s(d) \alpha} \# \widetilde \fR\prod_{j=1}^{d}\zeta_j \ll \int_{\Tor} |W_{d}( \vx; N)|^{2s(d)}\, d\vx\le N^{s(d)+o(1)},
$$
which yields  the desired bound.  
\end{proof}

Note that the above bound of $\# \widetilde \fR$ is nontrivial when $1/2<\alpha<1$. 

From Corollary~\ref{cor:key} and Lemma~\ref{lem:counting} 
we formulate  the following Corollary~\ref{cor:largebox}  for the convenience of our applications.  

\begin{cor}  
\label{cor:largebox}
Let $0<\alpha<1$ and $N_i=2^i, i \in \N$. Then for any  $\eta>0$  we have 
$$
\cE_{d, \alpha+\eta} \subseteq \bigcap_{k=1}^{\infty}\bigcup_{i=k}^{\infty} \bigcup_{\cR\in \fR(i)} \cR, 
$$
where each $\cR$ of $\fR(i)$ has the side length $\bzeta=(\zeta_1, \ldots, \zeta_d)$ such that 
$$
\zeta_j=1/ \rf{N_i^{j+1+\eps-\alpha}}, \quad j=1, \ldots, d,
$$
and furthermore 
$$
\#  \fR(i)\le N_i^{s(d)-2\alpha s(d)}\prod_{j=1}^{d} \zeta_j^{-1}\le N_i^{2s(d)(1-\alpha)+d(1-\alpha)+d\eps+o(1)}.
$$
\end{cor}

\section{Proof of Theorem~\ref{thm:dim}}
We start from some auxiliary results. First, we adapt the definition of the \emph{singular value  function} from~\cite[Chapter~9]{Falconer} to the following. 

\begin{definition} Let $\cR\subseteq \R^d$ be a rectangle with side lengths 
$$
r_1\ge \ldots \ge r_d.
$$ 
For $0< t \le d$  we set  $$
\varphi_{0, t}(\cR)=r_1^{t},
$$
and for $k=1, \ldots, d-1$ we define 
$$
\varphi_{k, t}(\cR)=r_1\ldots r_{k}r_{k+1}^{t-k}.
$$
\end{definition}

Note that  for a rectangle $\cR\subseteq \R^2$ with the side length $r_1\ge r_2$ we have 
$$
\varphi_{k, t} (\cR)=
  \begin{cases}
   r_1^t  & \text{ for } k=0, \\
   r_1 r_2^{t-1} &  \text{ for }  k=1.
  \end{cases}
$$

\begin{remark}
\label{rem:rem}
The notation  $\varphi_{k, t}(\cR)$ roughly means that we can cover the rectangle $\cR$ by 
about (up to a constant factor) 
$$
\frac{r_1}{r_{k+1}}\ldots \frac{r_k}{r_{k+1}}
$$
balls  
of radius
$ r_{k+1}$, and hence this leads to the term 
$$
\varphi_{k, t}(\cR) = \frac{r_1}{r_{k+1}}\ldots \frac{r_k}{r_{k+1}} r_{k+1}^t
$$
in the expression for the Hausdorff measure with the parameter $t$ 
(again up to a constant factor which does not affect our results). 
\end{remark}

From the definition of the Hausdorff dimension, using the above notation,  we have the following inequality 
\begin{equation}\label{eq:upperbound}
\begin{aligned}
\dim \cE_{d, \alpha+\eps} \le \inf  \{ t>0:~\sum_{i=1}^{\infty} &\sum_{\cR\in \fR(i)}  \varphi_{k, t}(\cR)<\infty, \\
& \quad  \text{ for some } \ k=0, \ldots, d-1  \}.
\end{aligned}
\end{equation}

Now we turn to the proof of Theorem~\ref{thm:dim}.  For  $k=1, \ldots, d-1$ and $0<t\le d$ we have 
\begin{equation} 
\begin{split}
\label{eq:dimR0}
\sum_{\cR\in \fR(i)} &\varphi_{k, t}(\cR)=\# \fR(i)\zeta_{k+1}^{t-k} \prod_{j=1}^{k} \zeta_j   \\
&\le N_i^{2s(d)(1-\alpha)+d(1-\alpha)+d\eps+o(1)}\\
& \qquad\qquad \qquad \times\left (N_i^{\alpha-1-\eps-(k+1)}\right )^{t-k} \prod_{j=1}^{k} N_i^{\alpha-j-1-\eps}\\
&\le N_i^{2s(d)(1-\alpha)+d(1-\alpha)+d\eps+(t-k)(\alpha-k-2-\eps )+k(\alpha-1-\eps)-s(k)+o(1)}.
\end{split}
\end{equation}
Here and in the following we denote 
$$
s(k)=\frac{k(k+1)}{2}.
$$ 
We remark that~\eqref{eq:dimR0}  also holds for the case $k=0$, in which we have  $s(k)=0$.  
To be precise for $k=0$ we have 
$$
\sum_{\cR\in \fR(i)} \varphi_{0, t}(\cR)\le N_i^{2s(d)(1-\alpha)+d(1-\alpha)+d\eps+t(\alpha-2-\eps )+o(1)}.
$$
Applying \eqref{eq:upperbound} we conclude that  
$$
\dim \cE_{d, \alpha+\eta} \le t
$$
provided that the parameters  $\alpha, \rho,  k, t$ satisfy the following  further condition
$$
2s(d)(1-\alpha)+d(1-\alpha)+(t-k)(\alpha-k-2 )+k(\alpha-1)-s(k)<0,
$$
which becomes 
$$
t>\frac{2s(d)(1-\alpha)+d(1-\alpha)+s(k)}{k+2-\alpha}.
$$
By the arbitrary choice of $\eta$ we finish the proof.


\section*{Acknowledgement}
%
%
%

This work was  supported   by ARC Grant~DP170100786.


\begin{thebibliography}{99}
%
%
%
%
%
%
\bibitem{BDG} J. Bourgain, C. Demeter and L. Guth, 
`Proof of the main conjecture in Vinogradov's mean value theorem for degrees higher than three', 
{\it Ann.\ Math.\/}, {\bf 184} (2016), 633--682. 

%
%
%
%
%

\bibitem{ChSh1} C. Chen and I. E. Shparlinski,
`On large values of Weyl sums',  
{\it Preprint}, 2019, available at \url{https://arxiv.org/abs/1901.01551}. 

\bibitem{ChSh2} C. Chen and I. E. Shparlinski,
`New bounds of Weyl sums',  
{\it Preprint}, 2019, available at \url{https://arxiv.org/abs/1903.07330}. 

\bibitem{ChSh3} C. Chen and I. E. Shparlinski,
`Large and small values  of quadratic Weyl sums',  
{\it Preprin\/}, 2019, available at \url{https://arxiv.org/abs/1907.03101}.



\bibitem{Falconer} K. J. Falconer, {\it Fractal geometry: Mathematical foundations and applications\/},
John Wiley, 2nd Ed., 2003.

\bibitem{IwKow} H. Iwaniec and E. Kowalski,
{\it Analytic number theory\/}, Amer.  Math.  Soc.,
Providence, RI, 2004.

%





%
%
%
%


\bibitem{Wool2} T.~D.~Wooley,
`The cubic case of the main conjecture in Vinogradov's mean value theorem',  
{\it Adv.  in Math.\/}, {\bf  294} (2016), 532--561.

\bibitem{Wool3} T.~D.~Wooley, `Perturbations of Weyl sums', {\it Internat. Math. Res. Notices\/}, {\bf 2016} (2016), 2632--2646.


\bibitem{Wool5} T.~D.~Wooley, 
`Nested efficient congruencing and relatives of Vinogradov's mean value theorem',
{\it Proc. London Math. Soc.\/}, (to appear).

\end{thebibliography}
\end{document}